\documentclass[11pt,english]{amsart}
\usepackage[T1]{fontenc}
\usepackage[latin1]{inputenc}
\usepackage{verbatim}
\usepackage{amstext}
\usepackage{amsthm}
\usepackage{amssymb}

\makeatletter
\numberwithin{equation}{section}
\numberwithin{figure}{section}
\theoremstyle{plain}
\newtheorem{thm}{\protect\theoremname}
\theoremstyle{remark}
\newtheorem{rem}[thm]{\protect\remarkname}
\theoremstyle{plain}
\newtheorem{prop}[thm]{\protect\propositionname}

\usepackage{tikz}
\usepackage{pgfplots}

\makeatother

\makeatother

\usepackage{babel}
\providecommand{\propositionname}{Proposition}
\providecommand{\remarkname}{Remark}
\providecommand{\theoremname}{Theorem}

\begin{document}


\addtolength{\textwidth}{0mm}
\addtolength{\hoffset}{-0mm} 
\addtolength{\textheight}{0mm}
\addtolength{\voffset}{-0mm} 


\global\long\def\AA{\mathbb{A}}%
\global\long\def\CC{\mathbb{C}}%
 
\global\long\def\BB{\mathbb{B}}%
 
\global\long\def\PP{\mathbb{P}}%
 
\global\long\def\QQ{\mathbb{Q}}%
 
\global\long\def\RR{\mathbb{R}}%
 
\global\long\def\FF{\mathbb{F}}%

\global\long\def\DD{\mathbb{D}}%
 
\global\long\def\NN{\mathbb{N}}%
\global\long\def\ZZ{\mathbb{Z}}%
 
\global\long\def\HH{\mathbb{H}}%
 
\global\long\def\Gal{{\rm Gal}}%

\global\long\def\bA{\mathbf{A}}%

\global\long\def\kP{\mathfrak{P}}%
 
\global\long\def\kQ{\mathfrak{q}}%
 
\global\long\def\ka{\mathfrak{a}}%
\global\long\def\kP{\mathfrak{p}}%
\global\long\def\kn{\mathfrak{n}}%
\global\long\def\km{\mathfrak{m}}%

\global\long\def\cA{\mathfrak{\mathcal{A}}}%
\global\long\def\cB{\mathfrak{\mathcal{B}}}%
\global\long\def\cC{\mathfrak{\mathcal{C}}}%
\global\long\def\cD{\mathcal{D}}%
\global\long\def\cH{\mathcal{H}}%
\global\long\def\cK{\mathcal{K}}%

\global\long\def\cF{\mathcal{F}}%
 
\global\long\def\cI{\mathfrak{\mathcal{I}}}%
\global\long\def\cJ{\mathcal{J}}%

\global\long\def\cL{\mathcal{L}}%
\global\long\def\cM{\mathcal{M}}%
\global\long\def\cN{\mathcal{N}}%
\global\long\def\cO{\mathcal{O}}%
\global\long\def\cP{\mathcal{P}}%
\global\long\def\cR{\mathcal{R}}%
\global\long\def\cS{\mathcal{S}}%
\global\long\def\cW{\mathcal{W}}%

\global\long\def\cQ{\mathcal{Q}}%
\global\long\def\kBS{\mathfrak{B}_{6}}%
\global\long\def\kR{\mathfrak{R}}%
\global\long\def\kU{\mathfrak{U}}%
\global\long\def\kUn{\mathfrak{U}_{9}}%
\global\long\def\ksU{\mathfrak{U}_{7}}%

\global\long\def\a{\alpha}%
 
\global\long\def\b{\beta}%
 
\global\long\def\d{\delta}%
 
\global\long\def\D{\Delta}%
 
\global\long\def\L{\Lambda}%
 
\global\long\def\g{\gamma}%
\global\long\def\om{\omega}%

\global\long\def\G{\Gamma}%
 
\global\long\def\d{\delta}%
 
\global\long\def\D{\Delta}%
 
\global\long\def\e{\varepsilon}%
 
\global\long\def\k{\kappa}%
 
\global\long\def\l{\lambda}%
 
\global\long\def\m{\mu}%

\global\long\def\o{\omega}%
 
\global\long\def\p{\pi}%
 
\global\long\def\P{\Pi}%
 
\global\long\def\s{\sigma}%

\global\long\def\S{\Sigma}%
 
\global\long\def\t{\theta}%
 
\global\long\def\T{\Theta}%
 
\global\long\def\f{\varphi}%
 
\global\long\def\ze{\zeta}%

\global\long\def\deg{{\rm deg}}%
 
\global\long\def\det{{\rm det}}%

\global\long\def\Dem{Proof: }%
 
\global\long\def\ker{{\rm Ker}}%
 
\global\long\def\im{{\rm Im}}%
 
\global\long\def\rk{{\rm rk}}%
 
\global\long\def\car{{\rm car}}%
\global\long\def\fix{{\rm Fix( }}%

\global\long\def\card{{\rm Card }}%
 
\global\long\def\codim{{\rm codim}}%
 
\global\long\def\coker{{\rm Coker}}%

\global\long\def\pgcd{{\rm pgcd}}%
 
\global\long\def\ppcm{{\rm ppcm}}%
 
\global\long\def\la{\langle}%
 
\global\long\def\ra{\rangle}%

\global\long\def\Alb{{\rm Alb}}%
 
\global\long\def\Jac{{\rm Jac}}%
 
\global\long\def\Disc{{\rm Disc}}%
 
\global\long\def\Tr{{\rm Tr}}%
 
\global\long\def\Nr{{\rm Nr}}%

\global\long\def\NS{{\rm NS}}%
 
\global\long\def\Pic{{\rm Pic}}%

\global\long\def\Km{{\rm Km}}%
\global\long\def\rk{{\rm rk}}%
\global\long\def\Hom{{\rm Hom}}%
 
\global\long\def\End{{\rm End}}%
 
\global\long\def\aut{{\rm Aut}}%
 
\global\long\def\SSm{{\rm S}}%

\global\long\def\psl{{\rm PSL}}%
 
\global\long\def\cu{{\rm (-2)}}%
 
\global\long\def\mod{{\rm \,mod\,}}%
 
\global\long\def\cros{{\rm Cross}}%
 
\global\long\def\nt{z_{o}}%

\global\long\def\co{\mathfrak{\mathcal{C}}_{0}}%

\global\long\def\ldt{\Lambda_{\{2\},\{3\}}}%
 
\global\long\def\ltd{\Lambda_{\{3\},\{2\}}}%
\global\long\def\lldt{\lambda_{\{2\},\{3\}}}%

\global\long\def\ldq{\Lambda_{\{2\},\{4\}}}%
 
\global\long\def\lldq{\lambda_{\{2\},\{4\}}}%
 
\title{the cuspidal cubic and line arrangements with only triple points}
\author{Xavier Roulleau}
\begin{abstract}
We describe a new infinite family of line arrangements in the projective
plane with only triple points singularities and recover previously
known examples. 
\end{abstract}

\maketitle

\section{Introduction}

Line arrangements, i.e. finite union of lines in the projective plane,
are studied in various domains: algebra, topology, combinatorics...
A double (respectively triple) point on a line arrangement is a point
where exactly $2$ (respectively $3$) lines of the arrangement meet.
The celebrated Sylvester-Gallai Theorem assert that any line arrangement
over the real field has at least one double point, or is the trivial
example: the union of lines passing through the same point. 

Hirzebruch \cite{Hirzebruch} then proved that over the complex field,
a line arrangement is either the trivial example or possesses double
or triple points. There exists up to now a unique non-trivial example
of a complex line arrangement with only triple points. Since then,
the construction of line arrangements with many or only triple points
attracted attention see e.g. \cite{HKS} for a historical account.
In the present paper, we give a new construction, over the finite
fields, of line arrangements with only triple points. Line arrangements
with only triple points are also interesting for their combinatorics,
linked with the Steiner triple systems, which are set $T$ of subsets
of order $3$ of a finite set, such that any subset of order two is
contained in a unique subset $t\in T$.

Before presenting our results, let us introduce some notations and
recall some properties of line arrangements. If $\cL=\ell_{1}+\dots+\ell_{n}$
is a line arrangement labeled by $\{1,\dots,n\}$, the data $M(\cL)$
of the triples $\{i,j,k\}\subset\{1,\dots,n\}$ such that the lines
$\ell_{i},\ell_{j}$ and $\ell_{k}$ meet at a common point is called
the matroid associated to $\cL$; it encodes the combinatorics of
$\cL$. If $\cL$ has only triple points, $M(\cL)$ is a Steiner triple
system. 

If $\g$ is a projective transformation of the plane, one may define
$\g\cL=\g\ell_{1}+\dots+\g\ell_{n}$. Both line arrangements $\cL,\,\g\cL$
have the same combinatorics: $M(\cL)=M(\g\cL)$. Given a matroid $M$,
a labelled line arrangement $\cL$ such that $M(\cL)=M$ is said a
realization of $M$. Given a field $K$, there exists a scheme parametrizing
the realizations over $K$ and the quotient of these realizations
modulo the action of $\text{PGL}_{3}(K)$ is called the realization
space. One says that a line arrangement over a field $K$ is rigid
if $\cR(M)_{/K}$ is zero dimensional. 

The aim of the present paper is to describe a new infinite family
of line arrangements in the plane with only triple points, and to
study some of their associated realization spaces:
\begin{thm}
For any power $q$ of $3$, there exists a line arrangement $\cL_{q}$
of $q$ lines over $\FF_{q}$ with $\tfrac{q(q-1)}{6}$ triple points
and no other singularities. \\
Let $M_{q}=M(\cL_{q})$ be the matroid associated to $\cL_{q}$. For
$q=3^{n}$ and $n\geq3$, the matroid $M_{q}$ has no realizations
over any field of characteristic $\neq3$. \\
The realization space $\cR(M_{27})_{/\overline{\FF}_{3}}$ is $3$
dimensional with a two dimensional immersed component. The reduced
scheme of $\cR(M_{27})_{/\overline{\FF}_{3}}$ is an open sub-scheme
of a quadric in $\PP^{4}$.
\end{thm}

We obtain a line arrangement $\cL_{q}$ as the dual of the $q$ points
defined over $\FF_{q}$ on the smooth part of the cuspidal cubic $y^{2}z=x^{3}$. 

The previously known line arrangements with only triple points are:

a) The trivial example: three lines meeting at one point.

b) Over a field containing a primitive third root of unity (in particular
over any algebraically closed field of characteristic $\neq3$), one
has the $\text{Ceva}(3)$ line arrangement with 9 lines:
\[
(x^{3}-y^{3})(x^{3}-z^{3})(z^{3}-y^{3})=0,
\]
and $12$ triple points. It is the dual of the Hesse configuration,
i.e. of the $9$ flex points of a smooth cubic. In characteristic
$3$, there is also a line arrangement of $9$ lines having the same
associated matroid as $\text{Ceva}(3)$. It is the $9$ lines in $\PP^{2}(\FF_{3})$
not containing a given point $p$ in $\PP^{2}(\FF_{3})$. These line
arrangements are rigid. 

c) The Fano plane: the $7$ lines in $\PP^{2}(\text{\ensuremath{\FF}}_{2})$.
It is also a rigid line arrangement. 

d) For any $n\geq4$, the sequence of line arrangements $\cC_{2^{n}-1}'$
with $2^{n}-1$ lines defined over $\overline{\FF}_{2}$ constructed
in \cite{HKS}. Their construction is obtained by taking the dual
of a generic projection to $\PP^{2}$ of the $2^{n}-1$ points $\PP^{n-1}(\FF_{2})$
in $\PP^{n-1}$.  

e) A rigid example over $\FF_{11}$ with $19$ lines (and $57$ triple
points) constructed in \cite{HKS}.

It is an open question \cite{Urzua} whether or not a) and b) are
the unique examples of line arrangements over $\CC$ with only triple
points. 

For any $n\geq2$, we also obtain over $\overline{\FF}_{2}$ some
line arrangements $\cC_{2^{n}-1}$ with $2^{n}-1$ lines and only
triple points, by using the dual of the non-zero points in $\AA^{1}(\FF_{2^{n}})$
in the smooth part of the cuspidal cubic curve (which is isomorphic
to $\AA^{1}$). 

For $n\geq4$, we obtain that our line arrangement $\cC_{2^{n}-1}$
and the line arrangement $\cC_{2^{n}-1}'$ of example d) define isomorphic
matroids, thus $\cC_{2^{n}-1}$ is another construction of $\cC_{2^{n}-1}'$. 

Our construction of $\cC_{n}$ generalizes a bit the one of $\cC_{n}'$
since it holds for $n=2$ and $3$. The line arrangement $\cC_{3}$
is the trivial example, and $\cC_{7}$ is the Fano plane, example
c). About the line arrangement $\cL_{n}$: for $n=3$, it is the trivial
example, and $\cL_{9}$ is example b).   We checked that the $19$
points in the dual of example e) are not contained in a cubic curve.
This example is the only one known that does not appear to be related
to cubic curves.

\section{Line arrangements using the cuspidal cubic}

\subsection{The cuspidal cubic}

Let $C\hookrightarrow\PP^{2}$ be the cuspidal cubic over a field
$K$ with a flex point $O$ defined over $K$. Let $C^{*}$ be the
complement of the singular point. As for a smooth cubic curve, one
may define a composition law (denoted by $+$) on $C^{*}$ by chords
and tangent, so that $O$ is neutral. 
\begin{thm}
\label{thm:Silver}(See e.g \cite[Chapter III, Proposition 2.5]{Silverman}).
The composition law $+$ makes $C^{*}(K)$ an abelian group isomorphic
to $(K,+)$. Three points $a,b,c\in C^{*}(K)$ (identified with $K$)
are on a line if and only if 
\[
a+b+c=0\text{ in }K.
\]
\end{thm}

In particular the tangent to a point $a$ in $C^{*}$ cuts $C^{*}$
at a residual point $b$ such that $2a+b=0$. 

Suppose that $K$ is a finite field of order $\#K=q=p^{n}$, where
$p$ is the characteristic of $K$. From Theorem \ref{thm:Silver},
the number of lines containing three distinct points of $K$ equals
to the number of sets 
\[
S_{a,b}=\{a,b,-(a+b)\}
\]
which have order $3$. One has $|S_{a,b}|=3$ if and only if 
\[
a\neq b\text{ and }a\neq-(a+b)\text{ and }b\neq-(a+b)
\]
which is equivalent to 
\begin{equation}
a\neq b\text{ and }b\neq-2a\text{ and }2b\neq-a.\label{eq:UN}
\end{equation}

\subsection{Characteristic $3$}

\subsubsection{General construction}

Suppose that $\text{Char}(K)=3$. Let $a,b\in K$ and $S_{a,b}=\{a,b,-(a+b)\}$.
Since $-2=1$ in $K$, the conditions \eqref{eq:UN} in order to have
$\#S_{a,b}=3$ are reduced to $a\neq b$. Then there are then $\tfrac{1}{3}\left(\begin{array}{c}
q\\
2
\end{array}\right)=\tfrac{q(q-1)}{6}$ lines containing $3$ points and through each point in $K$ there
are $\tfrac{1}{2}(q-1)$ such lines. That gives a 
\[
\left(q_{\tfrac{1}{2}(q-1)},\,\left(\tfrac{q(q-1)}{6}\right)_{3}\right)
\]
configuration of points and lines. Let $\cL_{q}$ be the dual of the
$q$ points of $K$.
\begin{thm}
\label{thm:arrangement-Char3}The line arrangement $\cL_{q}$ of $q$
lines has $\tfrac{q(q-1)}{6}$ triple points and no other singularities. 
\end{thm}

\begin{proof}
The dual of the lines containing three points of $K$ are triple points
of the line arrangement $\cL_{q}$, thus $\cL_{q}$ has $\tfrac{q(q-1)}{6}$
triple points. A line arrangement with $q$ lines has at most $\tfrac{1}{3}\tfrac{q(q-1)}{2}$
triple points, and if that bound is reached there are no other singularities. 
\end{proof}
We remark that every point $a$ on $C^{*}$ is a flex: for any $a$
one has $3a=0$. Geometrically, the intersection of the tangent line
at $a$ with the cubic is three times the point $a$. 
\begin{rem}
As a group, or $\FF_{p}$-vector space, $\FF_{q}$ is isomorphic to
$(\FF_{p})^{n}$, where $q=p^{n}$. Since the relation $a+b+c=0$
is preserved under such an isomorphism, the group $GL_{n}(\FF_{p})$
induces symmetries on the matroid $M_{q}$ associated to $\cL_{q}$.
For $p=3$, since for any element $t$ one has $(a+t)+(b+t)+(c+t)=a+b+c$,
the relation $a+b+c=0$ is preserved by translation, thus the automorphism
group of $M_{3^{n}}$ is the general affine group of $\AA_{/\FF-3}^{n}$,
the semi-direct product of $(\ZZ/3\ZZ)^{n}$ with $GL_{n}(\FF_{3})$. 
\end{rem}

\subsubsection{Examples in characteristic $3$}

\noindent

$\bullet$ The line arrangement $\cL_{3}$ is $3$ lines through the
same point. 

$\bullet$ The line arrangement $\cL_{9}$ has $9$ lines and $12$
double points. The matroid associated to $\cL_{9}$ is the same as
the matroid associated to the Ceva(3) line arrangement (after suitable
labeling of the lines).

$\bullet$ %
Let $M_{27}$ be the matroid associated to the line arrangement $\cL_{27}$.
Using \cite{Oscar}, one computes that the realization space $\cR(M_{27})$
over $\overline{\FF}_{3}$ of $M_{27}$ is $3$ dimensional with a
non-reduced immersed component $Z$ of dimension $2$ and multiplicity
$2$, a model of it is an open sub-scheme of a rank $3$ quadric in
$\AA^{4}$. We give the equations for that model of $\cR(M_{27})$
in $\AA^{4}$ and the associated $27$ normal vectors of the line
arrangements in the appendix of the arXiv version of the paper.

Using the generic point of $\cR(M_{27})$, one can check using MAGMA
software that the 27 points of the dual of the generic realization
of $M_{27}$ are not on a cuspidal cubic. However for the generic
element of $Z$, (which is such that $Z_{\text{red}}$ is a sub-scheme
of $\cR(M_{27})$), the $27$ points dual to a realization of $M_{27}$
are on a cuspidal cubic curve. 
\begin{prop}
For any $q=3^{n}$ with $n\geq3$, the matroid $M_{q}$ has no realizations
over any field of characteristic $\neq3$ 
\end{prop}

\begin{proof}
Using \cite{Oscar}, one computes that the matroid $M_{27}$ has no
realizations over any field of characteristic $\neq3$. The matroid
$M_{q},\,q=3^{n}$ may also be described as triples $\{a,b,c\}$ in
$(\FF_{3})^{n}$ such that $a+b+c=0$. For $n\geq3$, the group $(\FF_{3})^{3}$
is a sub-group of $(\FF_{3})^{n}$, and a realization of $M_{q}$
gives a realization of $M_{27}$ by taking the labels indexed by that
sub-group. That forces the realization to be in characteristic $3$.
\end{proof}

\subsection{Characteristic $2$}

\subsubsection{General construction}

Suppose that $\text{Char}(K)=2$ and $K=\FF_{q}$, where $q=2^{n}$
for $n>1$. Condition in \eqref{eq:UN} for $S_{a,b}=\{a,b,-(a+b)\}$
to be of order $3$ are then $a\neq b$ and $a\neq0,\,b\neq0$. More
geometrically, since for any $a\in K$, one has $2a=0$, the tangent
line trough $a$ contains $0$, and the point $a$ with multiplicity
$2$, thus it contains no other points of the cuspidal cubic. That
implies that for any couple $a\neq b$ in $K^{*}$, the line going
through $a$ and $b$ meets the cubic in a third point $c=-(a+b)\not\in\{0,a,b\}$.
Through each point $a\neq0$ of $K$ there goes $\tfrac{1}{2}(q-2)$
lines containing $a$ and another point of $K^{*}$. There are $\tfrac{(q-1)(q-2)}{6}$
such lines, and each line contain $3$ points of $K^{*}$. One obtain
in that way a 
\[
\left((q-1)_{\tfrac{1}{2}(q-2)},\,\left(\tfrac{(q-1)(q-2)}{6}\right)_{3}\right)
\]
configuration of points and lines. Let $\cC_{q-1}$ be the dual of
the $q-1$ in $K^{*}$. We obtain
\begin{thm}
\label{thm:arrangement-Char2}The line arrangement $\cC_{q-1}$ of
$q-1$ lines has $\tfrac{(q-1)(q-2)}{6}$ triple points and no other
singularities. 
\end{thm}

\subsubsection{Examples in characteristic $2$}

\noindent

$\bullet$ The line arrangement $\cC_{2^{2}-1}$ is $3$ lines through
the same point. 

$\bullet$ The line arrangement $\cC_{2^{3}-1}$ is the Fano plane. 

$\bullet$ Let $N_{2^{n}-1}$ be the matroid associated to $\cC_{2^{n}-1}$;
this is the set of triples $\{a,b,a+b\}\subset\FF_{2^{n}}$ with $a\neq b,\,a\neq0,\,b\neq0$.
Using \cite{Oscar}, one computes that the realization space $\cR(N_{15})$
of realizations of $N_{15}$ over $\overline{\FF}_{2}$ is an open
subscheme in $\PP_{/\overline{\FF_{2}}}^{3}$, in particular it is
$3$ dimensional. We checked that the generic element in $\cR(N_{15})$
is such that there is a cuspidal cubic containing the $15$ points
dual to the $15$ lines.

For $n>3$, in \cite{HKS}, are constructed some line arrangements
$\cC_{2^{n}-1}'$ (over the field $\FF_{2^{3n-1}}$) with only triple
points and with the same number $q-1=2^{n}-1$ of lines as $\cC_{2^{n}-1}$.
The $2^{n}-1$ points dual to $\cC_{2^{n}-1}'$ are the image of the
$2^{n}-1$ points of $\PP^{n-1}(\FF_{2})$ in $\PP^{n-1}$ by a generic
projection onto $\PP^{2}$. 

In fact one can check that the line arrangements $\cC_{q-1}$ and
$\cC_{q-1}'$ define isomorphic matroids as follows. As a group, the
field $K=\FF_{q}$ (for $q=2^{n}$) is the group $(\FF_{2})^{n}$
and the set $(\FF_{2})^{n}\setminus\{0\}$ is contained in $\PP^{n-1}(\FF_{2})$.
Two distinct points $a,b$ of $(\FF_{2})^{n}\setminus\{0\}$ generate
a line in $\PP^{n-1}$ which contains the points $a,b$, $c=a+b$
(so that $a+b+c=0$, since we are in characteristic $2$), and no
other points of $\PP^{n-1}(\FF_{2})$. As shown in \cite{HKS}, a
generic projection $f:\PP^{n}\dashrightarrow\PP^{2}$ induces a bijection
between the set $\PP^{n-1}(\FF_{2})$ and its image and $f$ preserves
the alignments: for $a,b,c\in\PP^{n-1}(\FF_{2})$, one has $a+b+c=0$
in $(\FF_{2})^{n}$ if and only if $f(a),f(b)$ and $f(c)$ are on
a line in $\PP^{2}$. Thus a bijective morphism $\eta$ of groups
between $K$ and $(\FF_{2})^{n}$ defines a isomorphism between the
matroids $M(\cC_{q})$ and $M(\cC_{q}')$, i.e. sends bijectively
the triples $\{a,b,c\}$ of $N=N(\cC_{q})$ onto the triples of $N'=N'(\cC_{q}')$.
If $\cC_{q}'=(\ell_{a})_{a\in\FF_{2}^{n}}$ is a realization of $N'$,
then $(\ell{}_{\eta(a)})_{a\in K}$ is a realization of $\cC_{q}$.

The construction of $\cC_{15}'$ in \cite{HKS} explains why the realization
space $\cR(N_{15})$ is an open sub-scheme of $\PP_{/\overline{\FF}_{2}}^{3}$,
since the line arrangement $\cC_{15}'$ is obtained by the projection
to $\PP^{2}$ from a generic point in $\PP^{3}$ of the $15$ points
of $\PP^{3}(\FF_{2})$ in $\PP^{3}$.
\begin{rem}
For $n\geq3$, the projections $\PP^{n}\dashrightarrow\PP^{2}$ are
parametrized by the Grassmannian $G(n-2,n+1)$. For $n>4$, it would
be interesting to understand if as for $n=3$, an open sub-scheme
of the grassmannian is also the realization spaces of the realizations
of $N_{2^{n}-1}$. At least we checked that for $n=4$, the realization
space $\cR(N_{31})$ is $6$ dimensional and rational, as $G(2,5)$. 
\end{rem}

\begin{rem}
The line arrangements $\cC_{2^{n}-1}$ are defined over $\FF_{2^{n}}$:
this is an answer the question before Theorem 3.4 in \cite{HKS},
to find a smaller field over which $\cC'_{q}$ may be defined. 
\end{rem}

\subsection{Characteristic $\protect\neq2,3$}

Suppose that $\text{Char}(K)=p>0$ and $p\notin\{2,3\}$. Then for
$b=-2a$ and $b=-\tfrac{1}{2}a$, $S_{a,b}$ has order $2$. Thus,
the number of lines containing $3$ points and passing through $a\in C^{*}$
is $\tfrac{1}{2}(q-3)$, moreover the number of lines containing $3$
(distinct) points of $K$ is 
\[
\tfrac{1}{6}q(q-3)
\]
and the number of lines containing $2$ points is $q$. One has a
\[
\left(q_{\tfrac{1}{2}(q-3)},\,\left(\tfrac{q(q-3)}{6}\right)_{3}\right)
\]
configuration of points and lines. The dual is a line arrangement
of $q$ lines and $\tfrac{q(q-3)}{6}$ triple points. It has 
\[
\tfrac{q(q-1)}{2}-3\tfrac{q(q-3)}{6}=q
\]
 double points (corresponding to the $q$ $2$-rich lines in the dual
space).

\vspace{0.3cm} 

\noindent Xavier Roulleau\\
Université d'Angers, \\
CNRS, LAREMA, SFR MATHSTIC, \\
F-49000 Angers, France 

\noindent xavier.roulleau@univ-angers.fr
\end{document}